\documentclass{article}
\usepackage{amssymb,amsmath,amsthm,amsfonts,epsfig,graphicx}
\usepackage[utf8]{inputenc}
\usepackage[T1]{fontenc}
\usepackage[english]{babel}

\theoremstyle{plain}
\newtheorem{teo}{Theorem}[section]
\newtheorem{cor}[teo]{Corollary}
\newtheorem{lema}[teo]{Lemma}

\newtheorem{obs}[teo]{Remark}
\newtheorem{ejp}[teo]{Example}

\theoremstyle{definition}
\newtheorem{df}[teo]{Definition}

\DeclareMathOperator{\iso}{Iso}
\DeclareMathOperator{\acu}{Lim}
\DeclareMathOperator{\adj}{Adj}

\DeclareMathOperator{\diam}{diam}
\DeclareMathOperator{\clos}{clos}

\DeclareMathOperator{\dist}{dist}

\newcommand{\per}{\hbox{Per}}

\newcommand{\R}    {\mathbb R}

\newcommand{\Z}  {\mathbb Z}
\newcommand{\N}  {\mathbb N}
\renewcommand{\epsilon}{\varepsilon}

\author{Alfonso Artigue\footnote{
Departamento de Matemática y Estadística del Litoral, Salto-Uruguay}}
\title{Hyper-expansive homeomorphisms}

\begin{document}
\date{\today}
\maketitle
\begin{abstract}
A homeomorphism on a compact metric space is said hyper-expansive if every pair of different 
compact sets are separated by the homeomorphism in the Hausdorff metric. 
We characterize such dynamics as those with a finite number of orbits and whose non-wandering set is the union of the 
repelling and the attracting periodic orbits. 
We also give a characterization of compact metric spaces admiting hyper-expansive homeomorphisms.
\end{abstract}


\section{Introduction} 
It is an important goal in topological dynamics to understand the global 
behavior of expansive homeomorphisms. 
In light of the work of J. Lewowicz and K. Hiraide of expansive homeomorphisms of surfaces (see \cite{H,L}) it seems that the key point is to determine 
the topological properties of stable sets. 
On manifolds of arbitrary dimension it is proved in the above mentioned papers that
the topological dimension of stable sets is positive (i.e., contains non-trivial connected sets) 
and smaller than the dimension of the manifold (i.e., there are no stable points). 
But it seems that more technology is needed in order to understand the topological structure of such sets. 

A proof of the cited results is based in the following tool. 
Take an arc or a continuum on the manifold, iterate it with the 
homeomorphism and consider an accumulation point in the Hausdorff metric on compact subsets. 
So, it seems that is of interest to consider the dynamics of sets instead of single points.
This fact was noticed by H. Kato, who introduced the notion of continuum-wise expansiveness. 
Consider $f\colon X\to X$ a homeomorphism on a compact metric space.
We say that $f$ is \emph{continuum-wise expansive} if 
there is $\delta>0$ such that if $C\subset X$ is a compact connected set (i.e., continuum) such that 
$\diam(f^nC)<\delta$ for all $n\in\Z$ then $C$ is a singleton. 

One can try the following definition: a homeomorphism is \emph{compact-wise expansive} if there is $\delta>0$ such that if 
$C\subset X$ is compact and $\diam f^n(C)<\delta$ for all $n\in\Z$ then $C$ is a singleton. 
But it is easy to see that this is equivalent with expansiveness. One just has to notice that $\diam(\{x,y\})=\dist(x,y)$ and that 
every non-trivial compact set has at least two different points. 
Expansiveness is also equivalent with what could be called \emph{set-wise expansiveness} (with analogous definition and proof). 
It is interesting to remark that \emph{open-wise expansiveness} is some kind of sensitive dependence on initial conditions.

Given a compact metric space we consider the space of all compact subsets $A\subset X$. 
That space is called the \emph{hyperspace} of $X$ and is denoted as $2^X$. 
The topology of $2^X$ is defined by the \emph{Hausdorff metric} $\dist_H$ defined as 
\[
 \dist_H(A,B)=\inf\{\epsilon>0:A\subset B_\epsilon (B)\hbox{ and }B\subset B_\epsilon(A)\}
\]
for all $A,B\in2^X$. As usual $B_\epsilon (A)$ denotes the set $\cup_{x\in A} B_\epsilon(x)$ where 
$B_\epsilon(x)=\{y\in X:\dist(x,y)<\epsilon\}$.
The hyperspace has very nice properties. 
For example, it is known that $2^X$ inherits the compactness of $X$. 
Also, if $X$ is connected then $2^X$ is arc-wise connected (see \cite{Nadler}). 
So, it is natural to extend the action of $f$ to $2^X$, simply as $\hat f\colon 2^X\to 2^X$ defined by 
$\hat f(A)=f(A)$. It gives a homeomorphism as can be easily verified. 
Some relationships are known between the dynamics of $f$ and $\hat f$. 
For example, if $f$ has positive topological entropy then $\hat f$ has infinite topological entropy (see Proposition 6 in \cite{BaSi}). 
 
The purpose of the present note is to study the expansiveness of the induced map $\hat f$, 
that is what we call \emph{hyper-expansiveness} of $f$. 
Notice that hyper-expansiveness is a stronger condition than expansiveness.
The following facts are known:
\begin{itemize}
\item if a compact metric space admits an expansive homeomorphism then its topological dimension is finite (see \cite{Ma}) and
\item if $\dim_{top} X>0$ then $\dim_{top} 2^X=\infty$ (this fact was first proved in \cite{Maz}, see also \cite{Nadler} Theorem 1.95).
\end{itemize}
Hence, if $2^X$ admits an expansive homeomorphism then $\dim_{top}X=0$.

It is known that expansiveness does not imply hyper-expansiveness. Indeed, in \cite{BaSi} it is noticed that
the shift map is not hyper-expansive (this can be deduced from the fact that the
shift map has infinite periodic points) while the shift map itself is expansive.
Those remarks on hyper-expansiveness were rediscovered in \cite{Sharma} (Proposition 2.23 and Example 2.24).
We give a simple characterization of hyper-expansiveness, 
statements and proofs are in the following Section.

Another important problem in topological dynamics is to determine what spaces 
admit expansive homeomorphisms. 
In \cite{KP} this problem is solved for countable compact spaces. 
As we will see, spaces admiting a hyper-expansive homeomorphism are countable. 
In this note we also give a characterization of compact spaces admiting hyper-expansive homeomorphisms.

In terms of the hyperspace, expansiveness can be characterized as follows. 
Let $F_1=\{\{x\}:x\in X\}\subset 2^X$ be the space of singletons. 
Notice that $F_1$ is $\hat f$-invariant, in fact $\hat f\colon F_1\to F_1$ is conjugated with $f\colon X\to X$.
By definition we have that $f$ is expansive if and only if $F_1$ is an isolated set for $\hat f$, i.e., 
there is an open set 
$U$ of $2^X$ such that $F_1= \cap_{n\in \Z} \hat f^nU$.

I would like to thank Damián Ferraro, Mario González and 
Ignacio Monteverde for useful conversations on these topics, José Vieitez 
for his corrections in the preliminary version of the note and the referee for 
his or her remarks. 

\section{Hyper-expansiveness}
Let $(X,\dist)$ be a compact metric space.

\begin{df}
 A homeomorphism $f\colon X\to X$ on a compact metric space is \emph{hyper-expansive} 
if $\hat f\colon 2^X\to 2^X$ is expansive, that is, there is $\delta>0$ such that if $\dist_H(f^nA,f^nB)<\delta$ for all $n\in\Z$, with $A$ and $B$ compact subsets of $X$, then $A=B$.
\end{df}

We need some definitions.
Given a point $p\in X$ we say that it is \emph{(Lyapunov) stable} if for all $\epsilon>0$ there is $\delta>0$ such that 
if $\dist(x,p)<\delta$ then $\dist(f^nx,f^np)<\epsilon$ for all $n\geq 0$. 
A point $p$ is said to be \emph{unstable} if it is stable for $f^{-1}$. 
We say that $p$ is \emph{asymptotically stable} if it is stable and there is $\gamma>0$ such that if $\dist(x,p)<\gamma$ 
then $\dist(f^nx,f^np)\to 0$ as $n\to \infty$. 
If $p$ is an asymptotically stable periodic orbit then the orbit of $p$ is said to be an \emph{attractor}. 
A \emph{repeller} is an attractor for $f^{-1}$.
Notice that isolated periodic points are stable and unstable by definition.
Let us denote 
\begin{itemize}
 \item $\Omega f$ the set of non-wandering points, i.e., $x\in\Omega f$ if for all $\epsilon>0$ there is $n>0$ such that 
$f^n(B_\epsilon x)\cap B_\epsilon x\neq\emptyset$,
\item $\per_r$ the set of repeller periodic points and $\per_a$ the set of attracting periodic points.
\end{itemize}

Now we can state the main result of this note.

\begin{teo}\label{teoprinc}
A homeomorphism $f\colon X\to X$ is hyper-expansive if and only if 
$f$ has a finite number of orbits and $\Omega f=\per_r\cup\per_a$.
\end{teo}

\begin{obs}
 It is easy to see that every expansive homeomorphism has a finite number of fixed points. 
Also, every compact $f$-invariant set $K\subset X$ (i.e., $f(K)=K$) is a fixed point of $\hat f$. 
So, if $f$ is hyper-expansive then $f$ has a finite number of compact invariant sets (in particular, it has finitely many periodic points). 
\end{obs}

A compact $f$-invariant set $K\subset X$ is said to be \emph{minimal} if for all $x\in K$ the orbit $\{x,fx,\dots,f^nx,\dots\}$ is dense in $K$.

\begin{lema}\label{minper}
 If $f\colon X\to X$ is hyper-expansive and $K\subset X$ is minimal then $K$ is finite (i.e., a periodic orbit).
\end{lema}

\begin{proof}
 Minimality implies that for all $\epsilon>0$ there is $n\geq 0$ such that for all $x\in X$ the set 
 $O_nx=\{x,fx,\dots,f^nx\}$ is $\epsilon$-dense in $K$ (i.e., for all $y\in K$ there is $j\in\{0,1,\dots,n\}$ such that 
 $\dist(y,f^j x)<\epsilon$). 
Therefore, $f^j(O_nx)$ is $\epsilon$-dense for all $j\in \Z$ because 
$f^j(O_nx)=O_n(f^jx)$.
 If $\epsilon$ is an expansive constant for $\hat f$ then $\dist_H(f^j(O_nx),f^j(K))<\epsilon$ for all $j\in \Z$. 
 Then $K=O_nx$ and $K$ is finite.
\end{proof}

\begin{obs}
In the previous proof the expansiveness was contradicted with two sets $K_1\subset K_2$. 
Notice that $\dist_H(A,B)\geq\dist_H(A,B\cup A)$, so $f$ is hyper-expansive if and only if there is $\delta>0$ such that 
if $A\subset B$, $A,B \in2^X$ and $\dist_H(\hat f^n A,\hat f^n B)<\delta$ for all $n\in \Z$ then $A=B$.
\end{obs}

We have that if $f$ is hyper-expansive then $f$ has a finite number of periodic points. 
Eventually taking a power of $f$ we can suppose that every periodic point is a fixed point. 
Recall that if a homeomorphism is expansive then its non-trivial powers are expansive too.
In the following Lemma we will need the next well known result.

\begin{obs}\label{estasest}
If $f$ is expansive and $p$ is a stable (unstable) periodic point then $p$ is an attractor (repeller). 
It can be proved as follows. 
Without loss of generality we can suppose that $p$ is a fixed point. 
By contradiction suppose that $p$ is stable but it is not asymptotically stable. 
Let $\delta>0$ be an expansive constant of $f$.
Therefore, there is a point $q\in B_\delta(p)$ such that 
$f^n q\in B_\delta(p)$ for all $n\geq 0$ but the $\omega$-limit set of $q$ is not $\{p\}$. 
So $p$ and a point $p'\in\omega(q)$, $p'\neq p$, contradict the expansiveness of $f$. 
\end{obs}

\begin{lema}
 If $f$ is hyper-expansive then every fixed point of $f$ is an attractor or a repeller.
\end{lema}

\begin{proof}
 By contradiction suppose that $p$ is a fixed point of $f$ that is neither attractor nor repeller. 
 Since $p$ is not an attractor, $p$ is not stable (Remark \ref{estasest}). 
 So, there is $\epsilon>0$ and a sequence $x_n$ such that $x_n\to p$ as $n\to\infty$ and for some $k_n>0$, 
 $f^{k_n}(x_n)\notin B_\epsilon(p)$. Suppose that for all $k<k_n$, $f^k(x_n)\in B_\epsilon(p)$. 
 Assume that $a_n=f^{k_n-1}(x_n)$ converges to $a\in\clos B_\epsilon(p)$. 
 It is easy to see that $f^j(a)\to p$ as $j\to-\infty$ and $a\neq p$. 
 
 Similarly, using that $p$ is not unstable, one can prove that there is $b\neq p$ such that $f^j(b)\to p$ as $j\to\infty$. 
Let $\delta>0$ be an expansive constant for $\hat f$. 
Take $n\geq 0$ such that $f^m(b),f^{-m}(a)\in B_\delta(p)$ for all $m\geq n$. 
Let $A=\{f^n(b),f^{-n}(a)\}$ and $B=A\cup\{p\}$. 
So, $A\neq B$ and $\dist_H(f^nA,f^nB)<\delta$ for all $n\in \Z$. That contradicts the expansiveness of $\hat f$.
\end{proof}

Now we prove our main result.

\begin{proof}(of Theorem \ref{teoprinc})
 \emph{Direct}. 
 Suppose that $f$ is hyper-expansive. 
 We have proved that there is a finite number of periodic points. 
 So, eventually taking a power of $f$ we can suppose that every periodic point is in fact a fixed point. 
 If there are only fixed points, there is nothing to prove ($X$ is finite). 
 So, suppose that $x\in X$ is not a fixed point. 
 Consider the $\omega$-limit set $\omega(x)$. 
 It is a compact invariant set, therefore it contains a minimal set, say $K$. 
 We have proved that every minimal set is a periodic orbit, so, it is a fixed point $K=\{p\}$. 
 It is easy to see that $\omega(x)=\{p\}$, since $p$ must be an attractor. 
 In particular $x$ is a wandering point. Then, we have proved that $\Omega(f)=\per_a\cup\per_r$. 
 
 Now we will prove that there is a finite number of orbits. 
 It is easy to see that for all $\epsilon>0$ there is $N\geq 0$ such that 
 if $x\notin B_\epsilon(\Omega(f))$ then $f^jx,f^kx\in B_\epsilon(\Omega(f))$ if $j\leq-N$ and $k\geq N$. 

 If $f$ has an infinite number of orbits and $\epsilon>0$ is smaller than an expansive constant for $\hat f$, then $X\setminus B_\epsilon(\Omega(f))$ is a compact infinite set. 
 So, there are $x,y\notin B_\epsilon(\Omega(f))$ and $p,q\in \Omega (f)$ such that $\omega(x)=\omega(y)=\{p\}$ and 
 $\alpha(x)=\alpha(y)=\{q\}$. 
 Then, if $\dist(x,y)$ is small, this two points contradicts the expansiveness of $f$ (and hyper-expansiveness too). 
 This contradiction proves that there is a finite number of orbits. 
 
 \emph{Converse}. 
 Again, eventually taking a power of $f$, we can assume that every periodic point of $f$ is a fixed point. 
 Let $\delta_1>0$ be such that $\cap_{n\geq 0}f^n(B_{\delta_1}(\per_a))=\per_a$ 
 and $\cap_{n\leq 0}f^n(B_{\delta_1}(\per_r))=\per_r$.
 Take $x_1,\dots,x_n$ one point of each wandering orbit of $f$. 
 Let $\delta_2>0$ be such that $B_{\delta_2}(x_i)=\{x_i\}$ for all $i=1,\dots, n$. 
 We will show that $\delta=\min\{\delta_1,\delta_2\}$ is an expansive constant for $\hat f$. 
 Let $A,B$ be two compact sets such that $\dist_H(f^nA,f^nB)<\delta$ for all $n\in \Z$. 
 If there is a wandering point $x$ such that $x\in A\setminus B$ then there is $k\in\Z$ and $i\in\{1,\dots,n\}$ such that 
 $f^kx=x_i$. So, $\dist_H(f^kA,f^kB)>\delta_2$. 
 This contradiction proves that the wandering points of $A$ and $B$ coincide. 
 If $A\neq B$ then there is a fixed point $p\in A\setminus B$ (similarly for $p\in B\setminus A$). 
 Without loss of generality suppose that $p$ is a repeller. 
 Since $p\notin B$ then there is $\epsilon>0$ such that $B_\epsilon(p)\cap B=\emptyset$.
 Take $n$ such that $B_{\delta_1} (p)\cap f^nB=\emptyset$. Since $p\in f^nA$ for all $n\in \Z$, we have that 
 $\dist_H(f^nA,f^nB)>\delta_1$, wich is a contradiction. So $f$ is hyper-expansive.
 \end{proof}

A simple consequence of the previous result is that if $X$ admits a hyper-expansive homeomorphism 
then $X$ is countable. As we will see, the converse is not true.
Let $$\iso(X)=\{x\in X:\hbox{there is } \epsilon>0\hbox{ such that } B_\epsilon(x)\cap X=\{x\}\}$$
and 
$$\acu(X)=X\setminus \iso(X).$$ 
The cardinality of a set $A$ is denoted as $|A|$.
\begin{teo}
\label{teocharact}
A compact metric space $X$ admits a hyper-expansive homeomorphism if and only if 
$2\leq |\acu(X)|<\infty$ or $\acu(X)=\emptyset$ (i.e., $X$ is finite).
\end{teo}

\begin{proof}
By Theorem \ref{teoprinc} we have that $\acu(X)\subset\Omega(f)$ that is because wandering points must be isolated. 
So, $\acu(X)$ is finite. If $X$ is infinite, there must be at least one attractor and one repeller, so $\acu(X)\geq 2$.

In order to prove the converse notice that if the set of limit points is finite then $X$ 
is countable.
Consider an infinite countable space $X$ (the finite case is trivial). 
Since every infinite continuum is uncontable, we have that $\dim_{top}(X)=0$. 
It is known that if $\dim_{top}(X)\leq n$ then $X$ is homeomorphic to a compact subset of $\R^{2n+1}$, 
see Theorem V2 in \cite{HW}. 
So, without loss of generality, we can assume that $X\subset \R$. 
Let $p_1<\dots<p_n\in X$, $n\geq 2$, be the limit points of $X$. 
We can also suppose that $X\subset [p_1,p_n]$ and for all 
$\epsilon>0$ we have that 
\begin{itemize}
  \item $X\cap (p_j,p_j+\epsilon)\neq\emptyset$ for all $j=1,\dots,n-1$ and 
  \item $X\cap (p_j-\epsilon,p_j)\neq\emptyset$ for all $j=2,\dots,n$
\end{itemize}

Define $I_j=X\cap (p_j,p_{j+1})$ for $j=1\dots,n-1$.
Now we define $f\colon X\to X$ as follows:
\begin{itemize}
 \item $f(p_j)=p_j$ for all $j=1,\dots,n$,
 \item if $x\in I_j$ and $j$ is odd then $f(x)$ is the first point of $X$ at the right of $x$ and
 \item if $x\in I_j$ and $j$ is even then $f(x)$ is the first point of $X$ at the left of $x$.
\end{itemize}
In this way $p_j$ is a repeller fixed point if $j$ is odd and it is an attractor if $j$ is even. 
So, by Theorem \ref{teoprinc} we have that $f$ is hyper-expansive.
\end{proof}

Since hyper-expansiveness is a very strong condition, we have that \emph{most} homeomorphisms 
satisfy the following result.

\begin{cor}
If $f\colon X\to X$ is a homeomorphism of a compact metric space $X$ and $|\acu(X)|=\infty$ then for all $\epsilon>0$ 
there are two different compact sets $A,B\subset X$ such that 
$$\dist_H(f^nA,f^nB)<\epsilon, \hbox{ for all } n\in\Z.$$
\end{cor}

It is a simple consequence of our previous result. 
It holds for example if $X$ is a manifold of positive dimension, 
a non-trivial connected space or a Cantor set.

Let us now give some examples and final remarks.

\begin{ejp}
Let $X=\{0\}\cup\{1/n:n\in \N\}$. 
Since $X$ has just one limit point we have that $X$ 
does not admit hyper-expansive homeomorphisms,
but it is easy to see that it 
admits an expansive one.
\end{ejp}

Countable compact spaces admiting expansive homeomorphisms can be characterized as follows. 
Recall that $\acu^{\lambda+1}(X)=\acu(\acu^\lambda(X))$, $\acu^1(X)=\acu(X)$ and
$$\acu^\lambda(X)=\bigcap_{\alpha<\lambda} \acu^\alpha(X)$$
for every limit ordinal number $\lambda$. 
The \emph{limit degree} of $X$ is the ordinal number 
$d(X)=\lambda$ if $\acu^\lambda(X)\neq\emptyset$ and $\acu^{\lambda+1}(X)=\emptyset$. 
In \cite{KP} (Theorem 2.2) it is shown that a countable compact space $X$ admits an expansive homeomorphism 
if and only if $d(X)$ is not a limit ordinal number. 

\begin{obs}
 Applying Theorem \ref{teocharact} we have that $X$ admits a hyper-expansive homeomorphism if and only if 
$d(X)\leq 1$ and $|\acu(X)|\neq 1$.
\end{obs}

It seems to be of interest to provide an example of a countable compact space do not admiting expansive homeomorphisms.

\begin{ejp}
Given $A \subset \R$ we say that $(a,b)\in A\times A$ is an \emph{adjacent pair} if 
there are no points of $A$ in the open interval $(a,b)$. The set of adjacent pairs is denoted as
$$\adj(A)=\{(a,b)\in A\times A:a<b, (a,b)\cap A=\emptyset\}.$$ 
Let $A_0=\{0\}\cup\{1/n:n\in \N\}$ and 
$$A_{n+1}=A_n\cup\bigcup_{(a,b)\in\adj(A_n\cap [0,1/n])}\{a+(b-a)/m:m\in\N\}.$$ 
Define $X=\cup_{n=0}^\infty A_n$. It is easy to see that it is a compact set and it is countable by construction. 
Notice that $d(X)$ is the first infinite ordinal number and therefore it is a limit ordinal number. 
To see that it does not admit an expansive homeomorphism notice that if $f\colon X\to X$ is a homeomorphism 
then $\acu^\lambda(X)$ is an $f$-invariant set for all ordinal number $\lambda$. 
Now notice that for all $\epsilon>0$ there is a finite ordinal number $\lambda$ such that 
$\acu^\lambda(X)\subset [0,\epsilon]$ and $\acu^\lambda(X)$ is an infinite set. 
Therefore, every pair of different points $x,y\in \acu^\lambda(X)$ contradict the $\epsilon$-expansiveness of $f$. 
Since $\epsilon$ is arbitrary we have that $X$ does not admit expansive homeomorphisms. 
\end{ejp}

\end{document}